\documentclass[11pt]{article}
\usepackage[margin=1in]{geometry}
\usepackage[tiny,compact]{titlesec}

\makeatletter
\def\blfootnote{\xdef\@thefnmark{}\@footnotetext}
\makeatother

\usepackage{algorithm}
\usepackage{algorithmic}

\usepackage{xspace}
\usepackage{url}
\usepackage{dgleich-setup}
\usepackage[nomathtools]{dgleich-math}

\usepackage{paralist}
\usepackage{epstopdf}
\newcommand{\Tr}{{T}}

\newcommand{\tA}{\mat{\tilde{A}}}
\newcommand{\tb}{\vec{\tilde{b}}}
\newcommand{\tx}{\vec{\tilde{x}}}

\newcommand{\tlambda}{{\widetilde{\lambda}}}

\newcommand{\sC}{{\mathcal{C}}}
\newcommand{\sF}{{\mathcal{F}}}
\newcommand{\sI}{{\mathcal{I}}}
\newcommand{\sP}{{\mathcal{P}}}

\title{Erasure coding for fault oblivious linear system solvers}
\author{David F.~Gleich, Ananth Grama, and Yao Zhu
\\
Purdue University, Department of Computer Science}

\begin{document}

\date{}
\maketitle

\blfootnote{~}
\begin{abstract} 
Dealing with hardware and software faults is an important problem
as parallel and distributed systems scale to millions of processing cores
and wide area networks. Traditional methods for dealing with faults include
checkpoint-restart, active replicas, and deterministic replay. Each of these
techniques has associated resource overheads and constraints. In this paper, we
propose an alternate approach to dealing with faults, based on input augmentation.
This approach, which is an algorithmic analog of erasure coded storage, applies a
minimally modified algorithm on the augmented input to produce an augmented output.
The execution of such an algorithm proceeds completely oblivious to faults
in the system. In the event of one or more faults, the real solution is recovered
using a rapid reconstruction method from the augmented output. We demonstrate this
approach on the problem of solving sparse linear systems using a conjugate gradient
solver. We present input augmentation and output recovery techniques. Through detailed
experiments, we show that our approach can be made oblivious to a large number of
faults with low computational overhead. Specifically, we demonstrate cases where
a single fault can be corrected with less than 10\% overhead in time, and 
even in extreme cases (fault rates of 20\%), our approach is able to compute a solution
with reasonable overhead. These results represent a significant improvement over
the state of the art.
\end{abstract}

\section{Introduction}

The next generation of parallel and distributed systems are projected to scale to
millions of processing cores and beyond. Distributed systems already span wide
geographic areas. In these regimes, hardware and software faults present major
challenges for scalable execution of programs. In this paper, we consider the
solution of an $n\times n$ nonsingular linear system
\begin{align}
\mA \vx&= \vb
\label{eq:raw-linear-system}
\end{align}
in an environment with faults.

Existing fault tolerant computations can be broadly classified as algorithmic or
system-supported. Algorithmic fault tolerance mechanisms alter the algorithm to
render it robust to faults. This requires deep algorithmic and analytical insight to
quantify fault tolerance properties and associated overheads. It can also
change the algorithmic cost and complexity significantly. System supported fault tolerance 
schemes include checkpoint-restart, active replicas, and deterministic replay.
Checkpoint-restart schemes involve the overheads of consistent checkpointing and I/O,
particularly for ultra-scale platforms, where I/O capacity and bandwidth are both at
a premium relative to the compute capability. Furthermore, the asynchronous nature of many
highly scalable algorithms makes it costly to identify consistent checkpoints and to
perform associated rollbacks. Variants of these schemes include in-memory checkpointing,
use of persistent storage (flash memory), and application-specified checkpoints.

Active replicas, commonly used in mission-critical real-time applications,
execute multiple replicas of each task. Failures are detected and replicas are
managed to support real-time constraints. While the runtime characteristics of
these schemes can be controlled (e.g., through worst-case runtime estimation), resource
overheads of such schemes are high, since tolerating a $k$-process failure
among $p$ processes requires $(k+1)p$ active processes. This cost is significant --
as an example, tolerating 10 faults in an ensemble of 1000 cores (a 1\% error rate)
requires 11,000 processes! Yet other systems such as MapReduce use the concept of
deterministic replay. In this model, computation proceeds in steps with checkpoints
at the end of each step. Processes are monitored within the steps, and in the event
of a failure, the computation associated with the failed process is replayed at an
alternate node. While this model has been successfully applied to a large set of
wide-area distributed applications, it has the drawbacks of staged execution and
increased job makespan, particularly when the number of faults is large. Furthermore,
checkpointing to persistent storage (typically a distributed file system) can add
significant overhead, particularly, when steps are small.

Our contribution in this paper is the design of a new type of method for
solving \eqref{eq:raw-linear-system} that we call a fault oblivious algorithm.
The essence of the idea is that we develop an algorithm that allows us to recover
the correct solution to the problem even if some of the computational units involved
in the solution process fail during the execution of the algorithm.  More specifically,
given system \eqref{eq:raw-linear-system}, we design an augmented system:
\begin{align}
\tA\tx&=\tb
\label{eq:encode-linear-system}
\end{align}
together with a solution strategy such that using the solution strategy to solve
\eqref{eq:encode-linear-system} in an environment with faults would have the following
properties:
\begin{enumerate}[1.]
\item \textit{Deterministic finite termination}. The solution process terminates in
finite steps. When it terminates, it indicates one and only one of the two cases:
(i) it fails to solve \eqref{eq:encode-linear-system} to a specified precision;
(ii) an approximate solution $\tx$ to \eqref{eq:encode-linear-system} within the
specified precision has been found.
\item \textit{Recoverability of the intended solution}. In case (ii) above, we
are able to recover the intended solution $\vx$ to \eqref{eq:raw-linear-system}
from $\tx$ through an inexpensive computation.
\end{enumerate}
As discussed above, a variety of existing fault tolerant techniques such as
check-pointing and replication give us finite termination and trivially satisfy
recoverability. Thus, our goal in designing the augmented system is that the
computational complexity of a solution process with deterministic finite termination
is bounded by the computational complexity of available fault-tolerant techniques.
In other words -- we want augmented systems that enable us to solve a given problem in
an environment with faults more quickly than existing techniques.

Our design for the augmented system \eqref{eq:encode-linear-system} is a linear
coding of the input system \eqref{eq:raw-linear-system}. For this reason, we refer
to \eqref{eq:encode-linear-system} as the \textit{encoded system}, and the input
system \eqref{eq:raw-linear-system} the \textit{raw system}. There are three major
components of our approach that we discuss in the next three sections. 
\begin{itemize}
\item The fault model -- Section~\ref{sec:fault}. We need to define the types of the
faults and their semantics that our coded linear solver can handle.
\item The encoding scheme -- Section~\ref{sec:encoding}. We add a set of rows
and columns to the matrix to render the new linear system singular, but consistent.
It remains this way for up to $k$ component-wise faults in the solution vector $\tx$.
\item The solution process and recovery scheme --
Section~\ref{section:solution-process-and-recovery-scheme}. Our solution process
is to run a conjugate gradient algorithm. We show that when this algorithm terminates,
it does so at a consistent solution of the encoded system. We then describe how to
recover the true solution $\vx$ in light of the encoding.
\end{itemize}
Because of the close relationship between the encoding scheme, solution process, and
recovery scheme, we adopt a co-design approach for these three tasks. 

For simplicity, in this manuscript, we restrict ourselves to the case where $\mA$
is symmetric positive definite (SPD). We report experiments on the efficiency of
our encoding in Section~\ref{sec:experiments}. Our main findings are that the
encoded system takes minimal additional work to solve in the presence of a fault.
In the presence of a substantial number of faults (20\% of components failing), it
takes $5$ times the number of iterations of a linear solver. However, note that each
iteration of this solver requires no additional overhead to achieve fault tolerance.
Thus, this is a substantial savings compared with a checkpoint-restart system. 

\section{Fault model} \label{sec:fault}

We view the execution of the algorithm in two stages -- the setup phase and the execution
phase. The setup phase consists of the input augmentation step of the algorithm. During
this phase, we assume that a small amount of reliable work can be done. In the presence of
faults, this can be achieved using more expensive fault tolerance techniques such as
replicated execution or deterministic replay. Since this step is a very small fraction of
the overall computation (less than 1\% for typical systems), the overhead is not 
significant. Please note that in current systems, the entire execution
is performed in this reliable model. Thus, we may assume that we have
this ability, although we wish to limit our use of these expensive schemes to minimize
performance overhead.

The execution phase of the algorithm corresponds to the solve over the augmented system.
In this phase, we assume an ensemble of message passing processes executing the solver.
During this phase of execution, we assume fail-stop failures; i.e., in the event of a
fault, a process halts. No further messages are received from this process by
any of the other processes. Indeed there are other fault models as well, ranging from
transient (soft) faults to Byzantine behavior. Soft faults manifest themselves
in the form of erroneous data. This data, when incorporated into data at other
processes, can lead to cascading error in programs.

Our proposed method can be extended to these other fault classes using existing
fault detection schemes. In such schemes, messages are signed with a checksum,
allowing us to detect on-the-wire errors. Asserts in the program, corresponding to
predicates whose violation signifies an error can be used to detect soft errors in
processes. If either of these soft efforts are detected, the receiving process
simply drops messages, thus emulating a fail-stop error. When a soft error is detected
at a process, the process is killed, once again resulting in a fail stop failure. Asserts
work similarly when Byzantine failures are detected. Thus, a combination of tolerance
to fail-stop failures with fault detection techniques allows us to deal with
a broad set of faults.

Please note that there are other system software related issues
associated with the proposed fault oblivious paradigm. Specifically, in many APIs a
single process failure can cause the entire program to crash. In yet
other scenarios, a crashed process can cause group communication operations
(reductions, broadcasts, etc.) to block. These kinds of program behavior would not allow
leveraging of our proposed schemes. We assume program behavior
in which faulty processes simply drop out of the ensemble, while the rest continue.
This paper does not address the design of such a fault oblivious API. Rather
it algorithmically establishes the feasibility and superior performance of
the erasure coded computation scheme for linear solvers.

\section{Encoding scheme} \label{sec:encoding}

Let $\vx^\ast$ be the true solution of the linear system $ \mA \vx = \vb$.
Let $k\leq n$ be the number of allowed faults during its execution.
Let $\mE\in\mathbb{R}^{n\times k}$ be an encoding matrix that we'll
specify completely shortly. We design the augmented matrix
$\tA\in\mathbb{R}^{(n+k)\times (n+k)}$
as follows:
\begin{align}
\tA &= \left[\begin{array}{ll}
\mA & \mA \mE\\
\mE^{\Tr} \mA & \mE^{\Tr} \mA \mE
\end{array}\right].
\label{eq:encode-matrix}
\end{align}
We choose the encoding of $\vx^\ast$ to be an embedding into $\mathbb{R}^{n+k}$, i.e.,
\begin{align}
\tx^{\ast}&=\left[\begin{array}{l}
\vx^{\ast}\\
0
\end{array}
\right]
\label{eq:encode-solution}
\end{align}
Accordingly, the encoding of $\vb$ is given by
\begin{align}
\tb &= \tA\tx^{\ast} = \left[\begin{array}{l}
\vb\\
\mE^{\Tr} \vb
\end{array}\right]
\end{align}

\subsection{Basic properties}

We now establish a few properties of these systems in terms of their rank,
a characterization of the solutions, and the semi-definiteness of $\tA$.

\begin{proposition} A null space basis of $\tA$ is $\left[\begin{array}{l}\mE\\ -\mI_{k}\end{array}\right]$.
\label{prop:null-space-basis}
\end{proposition}
\begin{proof}
From the design of $\tA$ in \eqref{eq:encode-matrix} and $\mA$ being SPD, we have $\rank(\tA)=n$. Thus the null space has dimension $k$. Then, by inspection, $\tA \sbmat{\mE \\ -\mI_k} = 0$ and $\sbmat{\mE \\ -\mI_k}$ has column rank $k$.
\end{proof}

As a corollary of the above proposition, we have the following proposition regarding the non-ambiguity of the solution encoding \eqref{eq:encode-solution}.
\begin{proposition}
Let $\sbmat{ \vy \\ \vz}$ be any solution of \eqref{eq:encode-linear-system} where $\vy \in \RR^{n}$. Once $\vz \in \RR^{k}$ is specified, then the components of $\vy$ are uniquely determined. Moreover, if $\vz = 0$, then $\vy = \vx^\ast$.
\label{prop:non-ambiguity}
\end{proposition}
\begin{proof}
Note that $\sbmat{\vx^\ast \\ 0}$ is a solution to \eqref{eq:encode-linear-system}. Thus, any solution to \eqref{eq:encode-linear-system} can be written as: 
\[ \bmat{ \vy \\ \vz } = \bmat{ \vx^\ast \\ 0 } + \bmat{ \mE \\ -\mI_k} \va \]
for a unique $\va \in \RR^{k}$. Due to the non-zero structure, we have $\va = -\vz$. Hence, $\vy$ is uniquely determined as $\vx^\ast - \mE \vz$. The final statement follows from $\vz = 0$.
\end{proof}

We now prove that $\tA$ as given in \eqref{eq:encode-matrix} is symmetric positive semidefinite (SPSD).
\begin{proposition}
If $\mA$ is symmetric positive definite, then $\tA$ as defined in \eqref{eq:encode-matrix} is SPSD.
\label{prop:encode-matrix-SPD}
\end{proposition}
\begin{proof}
Let the Cholesky factorization of $\mA$ be $\mA = \mL \mL^T$. Then, by inspection, we have 
\[ \tA = \bmat{ \mL \\ \mE^\Tr \mA \mL^{-T} } \bmat{ \mL \\ \mE^\Tr \mA \mL^{-T} }^T. \]
\end{proof}

\subsection{Solution degeneracies and faults}

The matrix $\tA$ has rank $n$, despite having $n + k$ rows and columns.
We now show how a specific use of this degeneracy allows us to have fault tolerant
solutions to \eqref{eq:encode-linear-system}. Let $\tx\in\mathbb{R}^{(n+k)}$ be the
encoded solution. For the same of clarity in presentation, we assume that faults can
\textit{only} occur within the components of $\tx_{1:n}$, i.e., the redundant
components $\tx_{(n+1):(n+k)}$ introduced by the encoding cannot be faulty.
Please note that this is not a limitation of our scheme.
We let the set of faulty indices be $\sF \subset [n]$. We constrain the
cardinality $|\sF| \leq k$. Let $\sC$ be the set of non-faulty (or correct)
components. Without loss of generality, we consider the system
\eqref{eq:encode-linear-system} in three components corresponding to the correct,
faulty, and redundant components of the solution. This is equivalent to a permutation,
after which we have that the solution is:
\begin{equation} \label{eq:permutation}
 \tx = \left[ \begin{array}{c} \vc \\ \vf \\ \vr \end{array} \right] \begin{array}{l} \text{correct} = \tx_{\sC} \\ \text{faulty} = \tx_{\sF} \\ \text{redundant}.  \end{array} 
\end{equation}
The overall permuted system is: 
\[ \bmat{ \mA_{11} & \mA_{12} & \mZ_1 \\
          \mA_{12}^T & \mA_{22} & \mZ_2 \\
					\mZ_1^T & \mZ_2^T & \mR } 
	 \bmat{ \vc \\ \vf \\ \vr }
	= 
	\bmat{ \vb_1 \\ \vb_2 \\ \mE^T \vb } 
	\text{ where } 
		\begin{cases} 
			\mZ_1 & = \mA_{11} \mE_1 + \mA_{12} \mE_2 \\
			\mZ_2 & = \mA_{12}^T \mE_1 + \mA_{22} \mE_2 \\
			\mR & = \mE^T \mA \mE 
		\end{cases}
\]
As our solver progresses, the components in $\vf$ become ``stuck'' at some
intermediate values as faults occur. We describe the semantics of these faults
more formally in the next section (Section~\ref{section:solution-process-and-recovery-scheme}).
The goal of this section is to show that we can recover solutions even when setting
$\vf$ to some arbitrary value.

Thus, for our erasure coded solvers, we need a condition on the matrix $\mE$ such that: 
\begin{compactenum}
\item There is always a solution to \eqref{eq:encode-linear-system} for any $\vf$ as long
as $|\sF| \le k$ (Proposition~\ref{prop:sufficient-condition-existence}).
\item Given any solution computed with faulty components ($|\sF| \le k$), we can extract
a solution to \eqref{eq:encode-linear-system}
(Proposition~\ref{prop:equivalence-encode-system-subsystem}).
\end{compactenum} 
The condition on the matrix $\mE$ that is essential to these results is the Kruskal rank.
Recall the definition: 
\begin{definition}[Kruskal rank~\cite{kruskal1977-three-way}]
The Kruskal rank, or $k$-rank, of a matrix is the largest number $k$ such that every
subset of $k$ columns is linearly independent.
\end{definition}
Notice that the condition for a matrix to be of Kruskal rank $k$ is much stronger than
being rank $k$. The following proofs require that the Kruskal rank of $\mE^T$ is $k$
in order to handle up to $k$ faults. Some intuition for this requirement is that
we need the matrix $\mE$ to encode redundancies to any possible faults with a number
up to $k$. Recovering the solution will require us to invert a matrix for the
components where the solution was faulty, and hence, we need all possible subsets
of $k$ rows of $\mE$ to be invertible -- giving us the Kruskal rank condition.
We now present these two results formally:

\begin{proposition} 
Let $\mE^T \in \RR^{k \times n}$ have Kruskal rank $k$ and let $\sF$ be an
arbitrary subset of $[n]$ with $|\sF| \le k$. Then there exists a solution to
\eqref{eq:encode-linear-system} with $\tx_{\sF} = \vf$ for any $\vf$.
When $|\sF|=k$, such a solution is unique.
\label{prop:sufficient-condition-existence}
\end{proposition}
\begin{proof}
Note that any solution of \eqref{eq:encode-linear-system} has the form: 
\[ \bmat{ \vx^* \\ 0 } + \bmat{ \mE \\ -\mI } \va \]
for some $\va \in \RR^{k}$. Let us permute this solution as in \eqref{eq:permutation}: 
\[ \bmat{ \vc \\ \vf \\ \vr } = \bmat{ \vx_1^* \\ \vx_2^* \\ 0 } + \bmat{ \mE_1 \\ \mE_2 \\ -\mI} \va. \]
It suffices to show that there exists $\va$ such that $\vf = \vx_2^* + \mE_2 \va$.
Because the rows of $\mE_2$ correspond to the faulty components, this is a set
of $|\sF|$ columns from $\mE^T$. These columns are linearly independent by the
Kruskal rank condition. Thus, there exists a solution to this underdetermined
linear system. If $|\sF| = k$, then the system is square and non-singular, so
the vector $\va$ is unique. 
\end{proof}

According to our fault model, as faults occur during an iterative process,
the components of $\vf$ become stuck (i.e., they are not updated because of lost messages).
Thus, the actual system that we solve is what we call a purified system consisting
of only non-faulty components:
\begin{equation} \label{eq:subsystem-purified}
\bmat{ \mA_{11} & \mZ_1 \\
       \mZ_1^T & \mR }
\bmat{ \vc \\ \vr }
= \bmat{\vb_1 \\ \mE^T \vb } - \bmat{ \mA_{12} \\ \mZ_2^T } \vf. 
\end{equation}			

By Proposition \ref{prop:sufficient-condition-existence}, if the encoding
matrix $\mE^T$ has Kruskal rank $k$, there exists a solution to the purified
subsystem \eqref{eq:subsystem-purified} from a solution to \eqref{eq:encode-linear-system}
with $\vx_{\sF} =\vf$ fixed. We now ask the reverse question.
Suppose $\sbmat{ \vc \\ \vr}$ is any solution to \eqref{eq:subsystem-purified},
will $\sbmat{ \vc\\ \vf \\ \vr}$ be a solution to \eqref{eq:encode-linear-system}
(under the permutation)?  The following proposition shows the answer is yes.
The reason we need this proof is that there are many possible solutions to
the purified subsystem. We need to establish that all solutions to
\eqref{eq:subsystem-purified} with $\vf$ fixed will lend us a full solution
to \eqref{eq:encode-linear-system}. 

\begin{proposition} \label{prop:equivalence-encode-system-subsystem}
Let $\mE^T \in \RR^{k \times n}$ have Kruskal rank $k$. Let $\sbmat{ \vc \\ \vr}$
be any solution to the purified system \eqref{eq:subsystem-purified},
where $|\sF| \le k$. Then $\sbmat{ \vc\\ \vf \\ \vr}$ is a solution to
\eqref{eq:encode-linear-system}.
\end{proposition}
\begin{proof}
This proof is equivalent to checking whether the following equation is satisfied by
the purified solution:
\begin{equation}
\bmat{ \mA_{12}^T & \mZ_2} \bmat{ \vc \\ \vr} = \vb_2 - \mA_{22} \vf.
\end{equation}
We establish this fact algebraically from the solution of the purified system.
First note that $\mE_2^T$ is a $k$-by-$|\sF|$ matrix with full column rank,
and thus it has a left-inverse $(\mE_2^T)^{\dagger} = (\mE_2^{} \mE_2^T)^{-1} \mE_2^{}$.
Now consider the two equations in the purified system: 
\begin{align}
\label{eq:pursys-1} \mA_{11} \vc + \mZ_1 \vr & = \vb_1 - \mA_{12} \vf \\
\label{eq:pursys-2} \mZ_1^T \vc + \mR \vr & = \mE_{}^T \vb - \mZ_2^T \vf.
\end{align}
The result of $\text{\eqref{eq:pursys-2}} - \mE_1^T \text{\eqref{eq:pursys-1}}$ is: 
\[ \mE_2^T \mA_{12}^T \vc + \mE_2^T \mZ_2^{} \vr = \mE_2^T \vb_2^{} - \mE_2^T \mA_{22} \vf. \]
To complete the proof, we premultiply this equation by left inverse $(\mE_2^T)^{\dagger}$.
\end{proof}

\section{The solution process and recovery scheme}
\label{section:solution-process-and-recovery-scheme}

Proposition \ref{prop:encode-matrix-SPD} establishes $\tA$ being SPSD when $\mA$ is SPD.
Thus we can apply the conjugate gradient (CG) method to solve a singular but consistent
linear system \eqref{eq:encode-linear-system} \cite{AshbyMS:1990}. For the $n+k$
eigenvalues of $\tA$, we have $0=\tlambda_1=\cdots=\tlambda_k<\tlambda_{k+1}\leq\cdots\leq\tlambda_{n+k}$.
Let the eigenvalues of $\mA$ be $0<\lambda_1\leq\cdots\leq\lambda_n$. Because of the
interlacing property, we have $\lambda_1\leq \tlambda_{k+1}$ and $\lambda_n\leq \tlambda_{n+k}$.
The effective condition number of $\tA$ is defined as $\kappa_{e}(\tA)=\frac{\tlambda_{n+k}}{\tlambda_{k+1}}$.
Specifically, we use the following two-term recurrence form of CG~\cite{Meurant:2006}.
\begin{algorithm}
\caption{}
\begin{algorithmic}[1]
\label{alg:two-term-recurrence-cg}
\STATE Let $\vx_0$ be the initial guess and $\vr_0=\vb-\mA \vx_0$, $\beta_0=0$.
\FOR{$t=0,1,\ldots$ until convergence}
\STATE $\beta_{t}^{}=(\vr_t,\vr_t)/(\vr_{t-1},\vr_{t-1})$
\STATE $\vp_t=\vr_t+\beta_t^{} \vp_{t-1}$
\STATE $\vq_t=\mA \vp_t$
\STATE $\alpha_{t}^{} =(\vr_t,\vr_t)/(\vq_t,\vp_t)$
\STATE $\vx_{t+1}=\vx_t+\alpha_t^{} \vp_t$
\STATE $\vr_{t+1}=\vr_t-\alpha_t^{} \vq_t$
\ENDFOR
\end{algorithmic}
\end{algorithm}

We consider the setting when Algorithm \ref{alg:two-term-recurrence-cg} is executed
in a distributed environment. For the encoded system \eqref{eq:encode-linear-system},
this means the encoded matrix $\tA$ and the encoded vectors are distributed among
multiple processes by rows. Let the index set associated with process $i$ be $\sI_i$,
then $[n+k]=\bigcup_{i}\sI_i$. According to our fault model, the operations of
Algorithm \ref{alg:two-term-recurrence-cg} affected by faults in a distributed
environment are the aggregation operations --- inner products and the matrix-vector
multiplication $\mA\vp_t$. Thus our erasure-coded CG can be defined by specifying
the semantics of these two aggregation operations under faults. At the $t$-th
iteration of erasure-coded CG, let the set of failed processes be $\sP_t$.
Then the set of faulty indices is $\sF_t=\bigcup_{i\in\sP_t}\sI_i$. We assume that 
each viable process can detect the breakdown of its neighbor processes. 
Based on this assumption, we specify the semantics of the two aggregation operations as follows
\begin{itemize}
\item Inner products $(\vr_t,\vr_t)$ and $(\vq_t,\vp_t)$. The viable processes carry out the all-reduce operation by skipping the faulty components $\sF_t$ in the vectors.
\begin{align}
(\vr_t,\vr_t) &= \left((\vr_t)_{[n+k]\backslash \sF_t},(\vr_t)_{[n+k]\backslash \sF_t}\right)
\label{eq:inner-product-fault-tolerant}
\\
(\vq_t,\vp_t) &= \left((\vq_t)_{[n+k]\backslash \sF_t},(\vp_t)_{[n+k]\backslash \sF_t}\right)
\nonumber
\end{align}
Furthermore, we require no reuse of aggregation operation results. This means that
when computing $\alpha_t$, $(\vr_t,\vr_t)$ is recomputed simultaneously
with $(\vq_t,\vp_t)$. Similarly when computing $\beta_t$, $(\vr_{t-1}, \vr_{t-1})$
is recomputed simultaneously with $(\vr_t,\vr_t)$. For this purpose, we have to
maintain both $\vr_t$ and $\vr_{t-1}$.
\item Matrix-vector multiplication $\vq_t=\mA \vp_t$. A viable process carries out its local aggregation operation for computing  $\mA_{\sI_i,:} \vp_t$ by skipping the faulty components $\sF_t$ in $\vp_t$.
\begin{align*}
\mA_{\sI_i,:}\vp_t &= \mA_{\sI_i,[n+k]\backslash \sF_t}(\vp_t)_{[n+k]\backslash \sF_t}
\end{align*}
\end{itemize}
Given the above semantics on aggregation operations, the erasure-coded CG on
$\tx$ can be effectively considered as an iterative solve process on the subsystem
defined on $\tx_{[n+k]\backslash \sF_t}$, as given in \eqref{eq:subsystem-purified}.
Note that the RHS of the purified subsystem \eqref{eq:subsystem-purified} depends on
$\vf$, the snapshot value of $\tx_{\sF_t}$. For this reason, we require each viable
process to cache the freshest snapshot values it have received from its neighbor
processes. Another technical issue we need to consider when faults happen is the
update to the search direction $\vp_t$. In fault-free CG with exact arithmetic,
we have $(\vr_t,\vp_{t-1})=0$. However, given the semantics of inner product
as defined in \eqref{eq:inner-product-fault-tolerant}, the orthogonality of
$\vr_t$ and $\vp_{t-1}$ will generally not hold. We truncate the
update $\vp_t=\vr_t+\beta_t \vp_{t-1}$ to be
\begin{align*}
\vp_t &= \vr_t
\end{align*}
whenever new faults occur before the next update to $\vp_t$.

Now we consider the recovery of the solution to the raw system \eqref{eq:raw-linear-system}.
Suppose the erasure-coded CG converges on the encoded system \eqref{eq:encode-linear-system}
after $T$ iterations. Let the encoding matrix $\mE^T\in\mathbb{R}^{k\times n}$ have
Kruskal rank $k$. Let $\sF\in [n]$ be the set of all faulty indices upon convergence
such that $|\sF|\leq k$. In erasure-coded CG, the snapshot value $\vf$ of the faulty
components $\tx_{\sF}$ are cached on the viable processes. Because of the semantics of
aggregation operations, the erasure-coded CG solves the purified subsystem defined
in \eqref{eq:subsystem-purified}. Let the returned approximate solution be
$\sbmat{ \vc \\ \vr}$. According to Proposition \ref{prop:equivalence-encode-system-subsystem},
then \[ \tx = \bmat{ \vc \\ \vf \\ \vr }. \] is a solution to the encoded system
\eqref{eq:encode-linear-system}. Then, by Proposition \ref{prop:null-space-basis}, we can
recover the intended solution to the raw system \eqref{eq:raw-linear-system} through the
equation
\begin{equation} \label{eq:recovery-equation}
\bmat{ \vx^* \\ 0 } = \tx + \bmat{ \mE \\ -\mI_k } \vr. 
\end{equation}
And by Proposition \ref{prop:non-ambiguity}, the recovery equation \eqref{eq:recovery-equation} is non-ambiguous.

\section{Related work}

The paper~\cite{BridgesFHH:2012} also considers the problem of designing a fault-tolerant
linear solver. They model the faults by encapsulating all fault prone operations as an
unreliable preconditioning operator and then build the fault-tolerant linear solver within
the framework of flexible Krylov subspace methods, like flexible GMRES~\cite{Saad:1993}.
The flexible outer iteration is required to be completely reliable. In contrast,
our erasure-coded approach does not require any part of the distributed environment
to be reliable. Another salient difference between~\cite{BridgesFHH:2012} and our
approach is on the kinds of tolerable faults. In~\cite{BridgesFHH:2012} the targeted
faults are value errors resulted from numerical operations, while our erasure-coded
approach targets more general fail-stop failures (subsystem breakdowns) in a distributed
environment. It's not hard to understand that the approach in~\cite{BridgesFHH:2012}
can not recover from such faults because of the permanent data loss when subsystems
breakdown in a distributed environment.
\par
The fault-tolerant linear solver in~\cite{BridgesFHH:2012} can be considered as a
natural application of the theory of inexact and flexible Krylov subspace
methods~\cite{EshofS:2004,SimonciniS:2003a,SimonciniS:2003b}. For inexact Krylov
subspace methods, the magnitude of perturbations need to be controlled to
maintain the convergence, although the perturbations can happen anywhere in the
result of a numerical operation. In contrast, our erasure-coded approach does not
presume any bound on the perturbations; but it does require the perturbations to
be sparse in the affected components. However, the theory of inexact and flexible
Krylov subspace methods can still provide the mathematical framework and tools to
guide the design and analysis of the erasure-coded linear solvers.

\section{Experimental results} \label{sec:experiments}

In this section, we report on experimental results with varying degrees of faults
and associated overhead. The main purpose of these experiments is to demonstrate the
feasibility of the idea of an erasure-coded linear solver. Through these experiments,
we intend to identify critical research problems to be solved in order to realize
the idea of an erasure-coded linear solver in a distributed setting.

\subsection{Experiment design}

In our experiments we uses three SPD test matrices, with their sizes ($n$), number of nonzeros ($nnz$), and types shown in Table \ref{table:test-matrix}. {\tt Ltridiag500} is a $500\times 500$ $1$D model matrix with the stencil $\left[\begin{array}{lll}
-1 & 2 & -1
\end{array}\right]$. {\tt mhdb416} and {\tt nos3} are from University of Florida Sparse Matrix Collection~\cite{DavisH:2011}.
We implemented a MATLAB code to simulate the erasure-code CG described in Section~\ref{section:solution-process-and-recovery-scheme}. In our current simulation, we inject fault components only at the end of a CG iteration. Furthermore, we assume all the faults happen at the same time. More extensive and thorough simulations that allow faults happen anywhere in a CG iteration and temporally interspersed are relegated to future work. For a given $k$ number of tolerable faults, we design the $n\times k$ encoding matrix $\mE$ as a random Gaussian matrix scaled by $\frac{1}{\sqrt{n}}$, i.e.,
\begin{align*}
\mE&=\frac{1}{\sqrt{n}}\bar{\mE}, \ \ \ \ \bar{\mE}_{ij} \sim \mathcal{N}(0,1), \ \ \ \ i=1,\ldots,n,\ j=1,\ldots,k
\end{align*}
This matrix has Kruskal rank $k$ with high probability.

\subsection{Experimental Results}

For each test matrix in Table \ref{table:test-matrix}, we run an erasure-coded CG simulation code for $k=0,1,20\%n$ number of faults. The RHS $\vb$ of the raw system \eqref{eq:raw-linear-system} are synthesized using random solution vectors $\vx$ in $(0,1)^n$. We set the number of maximum CG iterations to be $10n$, and monitor the convergence using the stopping criterion $\normof[2]{ \vr_t } \leq 10^{-10}.$ For each value of $k$, the $k$ fault components are randomly selected from $[n]$. These fault components are injected simultaneously at the end of one randomly selected CG iteration that is no larger than $0.25n$. We refer to this CG iteration as the fault point. In Figures \ref{figure:residual-norm-Ltridiag500} - \ref{figure:residual-norm-nos3}, we plot the residual norm vs.~CG iteration for the three test matrices; and from left to right for $k=0,1,20\%n$. The fault point is marked by the red cross.

\begin{table}[tb]
\caption{Test matrices}
\label{table:test-matrix}
\begin{center}
\begin{tabularx}{0.7\linewidth}{rXXl}
\toprule
  Matrix & $n$ & $nnz$ & Type \\
\midrule
  {\tt Ltridiag500} & $500$ & $1,498$ & $1$D model problem \\
  {\tt mhdb416} & $416$ & $2,312$ & electromagnetics problem\\
  {\tt nos3} & $960$ & $15,844$ & structural problem \\
\bottomrule
 \end{tabularx}
\end{center}

\end{table}

\begin{figure}[t]
  \centering 
  \includegraphics[width=0.32\hsize]{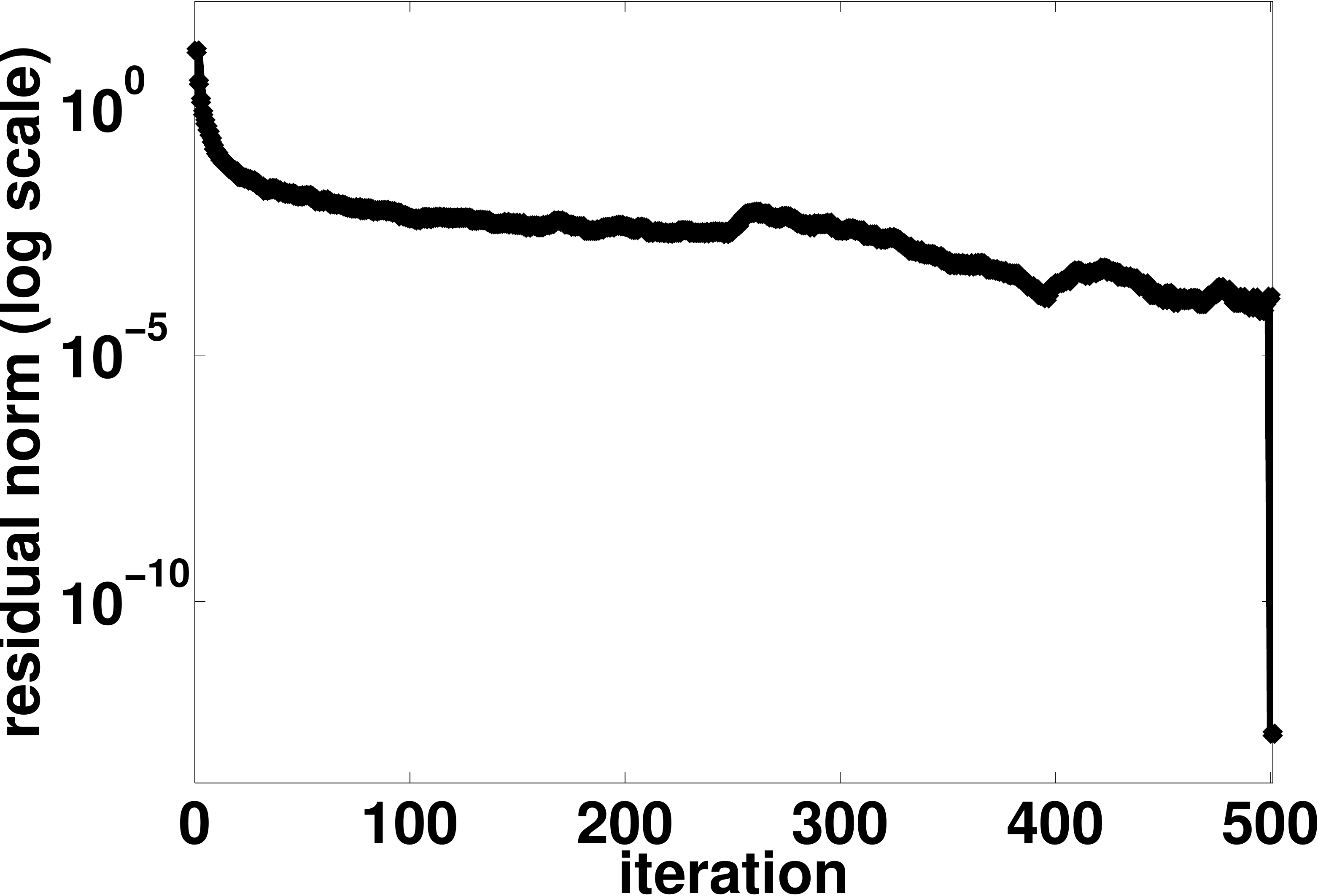}
  \includegraphics[width=0.32\hsize]{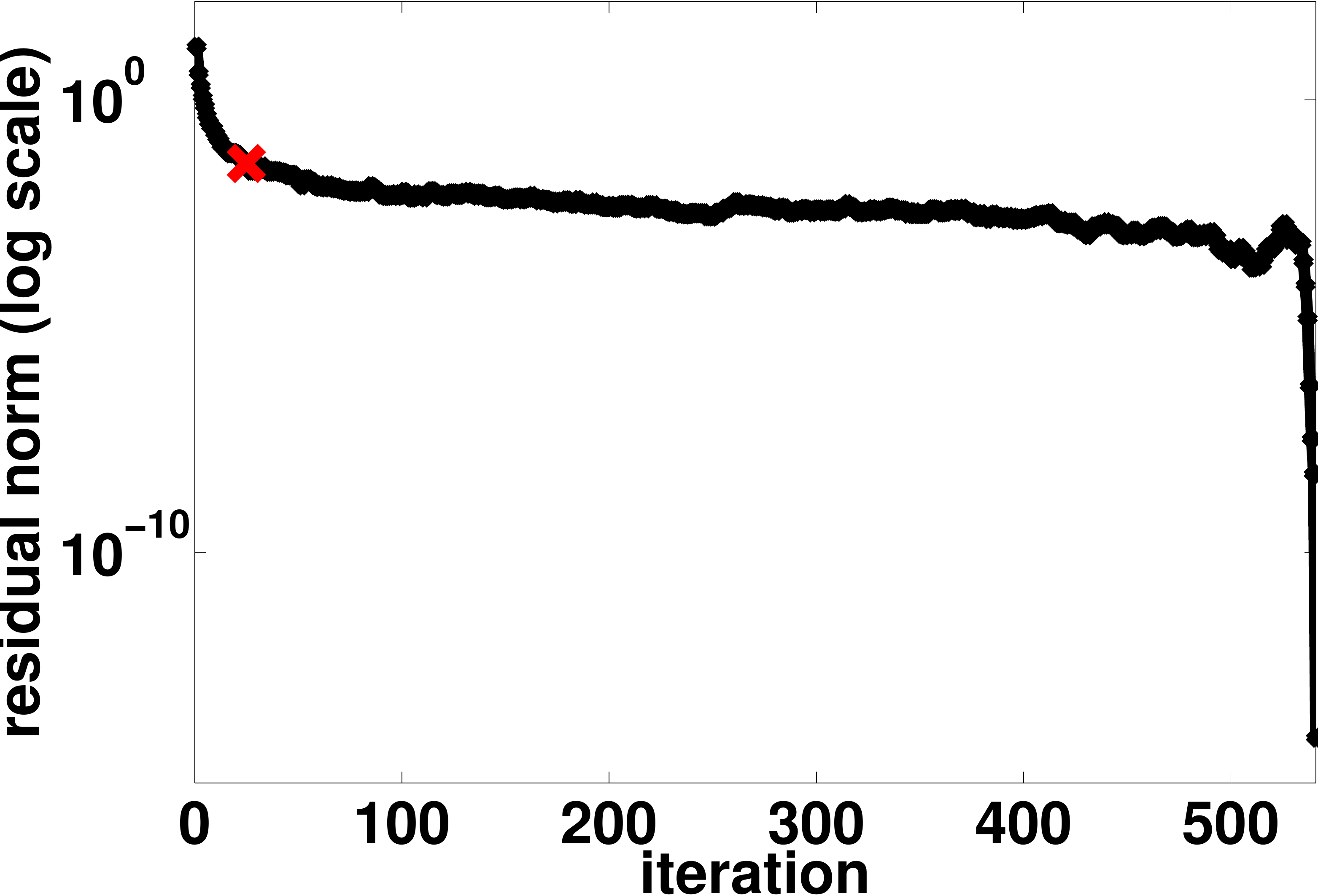}
  \includegraphics[width=0.32\hsize]{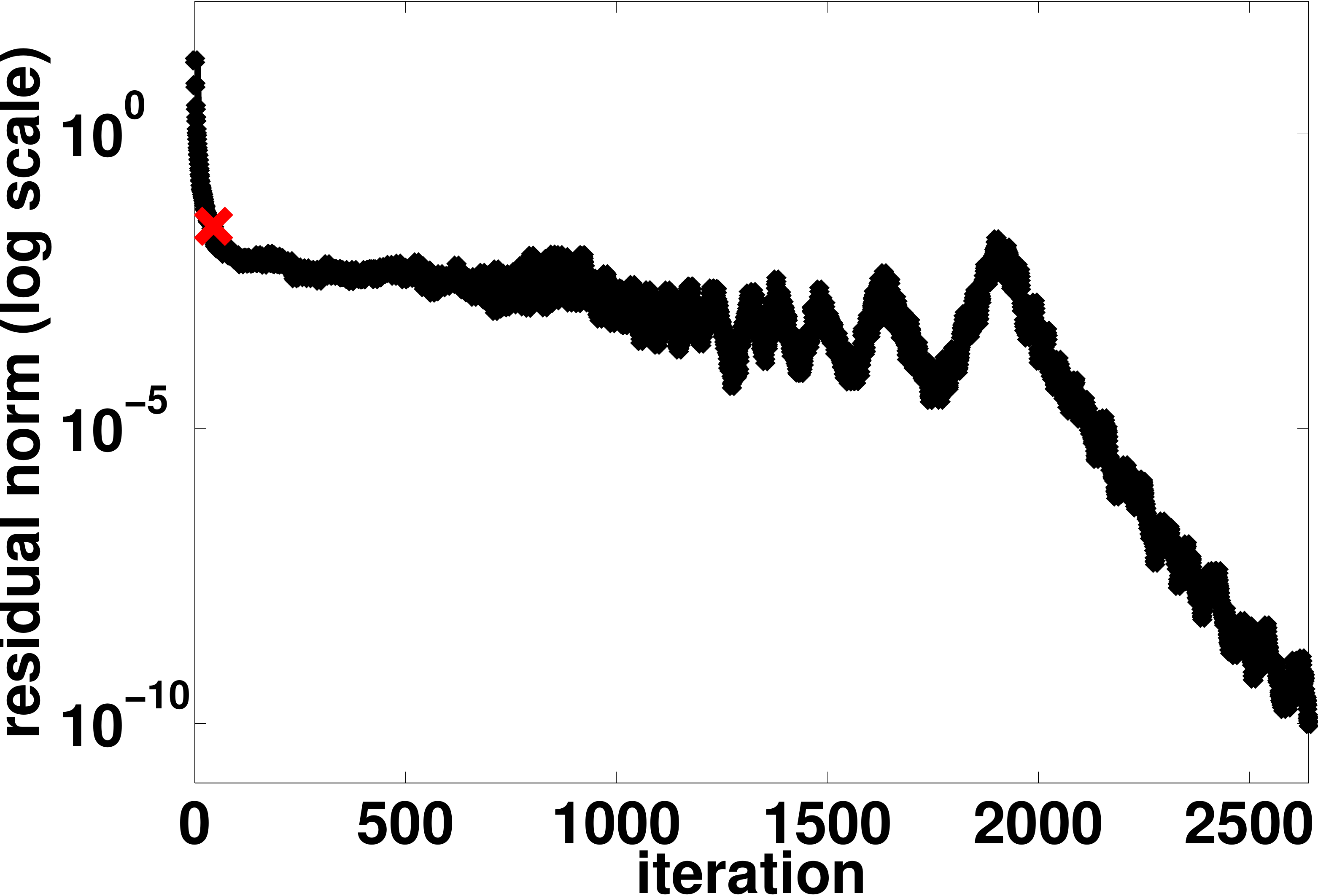}
  \caption{{{\tt Ltridiag500} residual norm. From left to right $k=0,1,20\%n$. The fault point is marked by the red cross.}}
  \label{figure:residual-norm-Ltridiag500}
  \vspace{\baselineskip}
  
  \includegraphics[width=0.32\hsize]{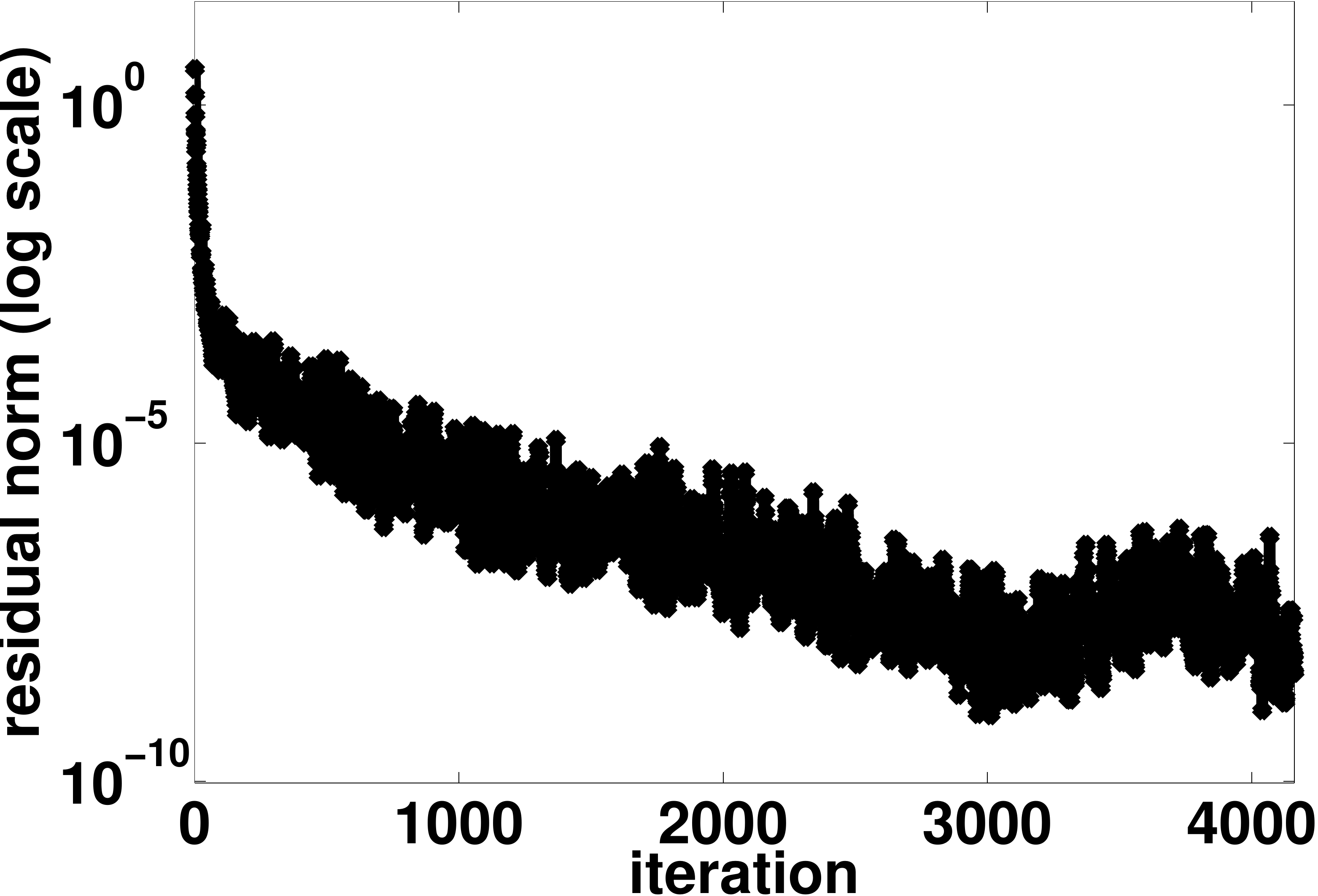}
  \includegraphics[width=0.32\hsize]{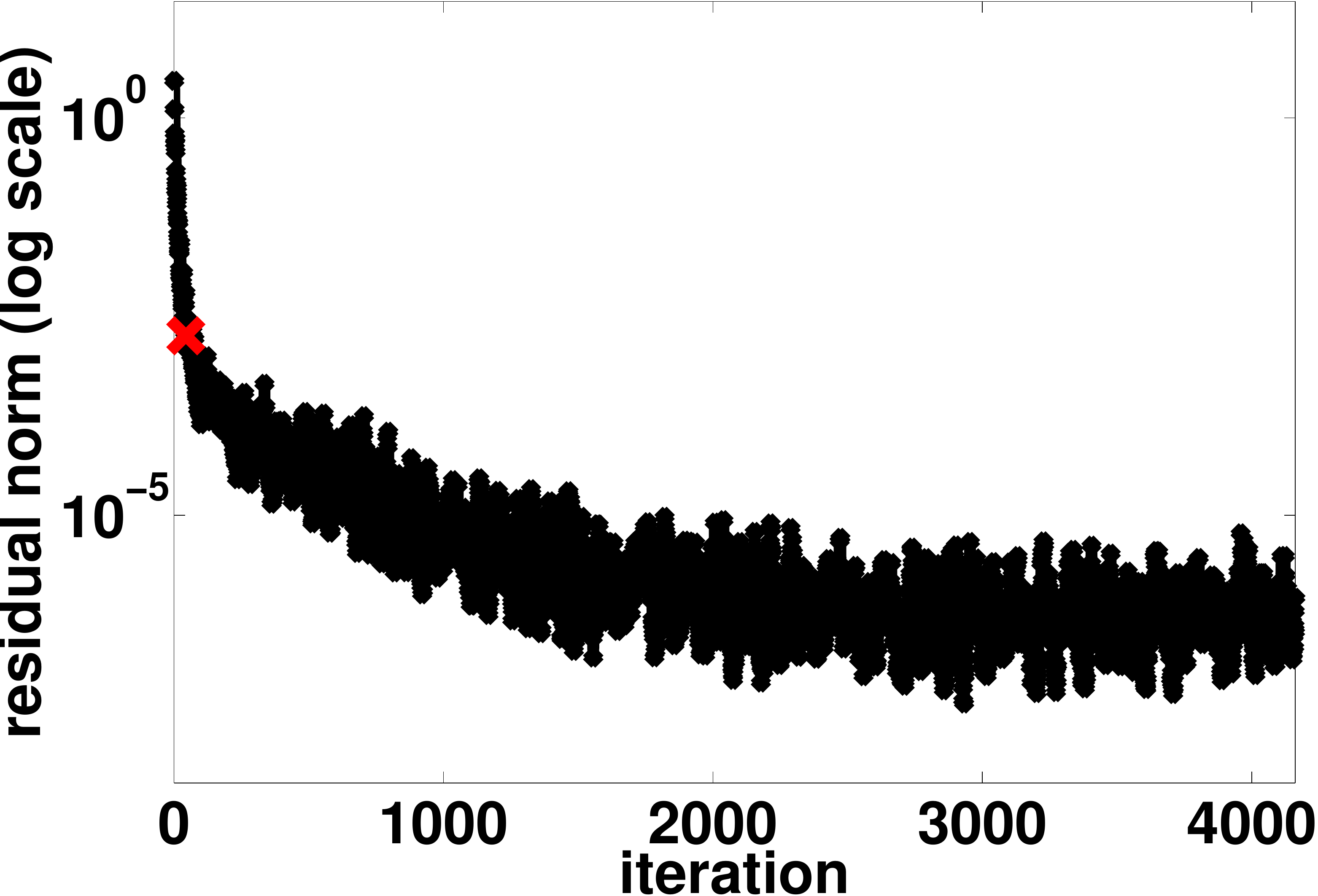}
  \includegraphics[width=0.32\hsize]{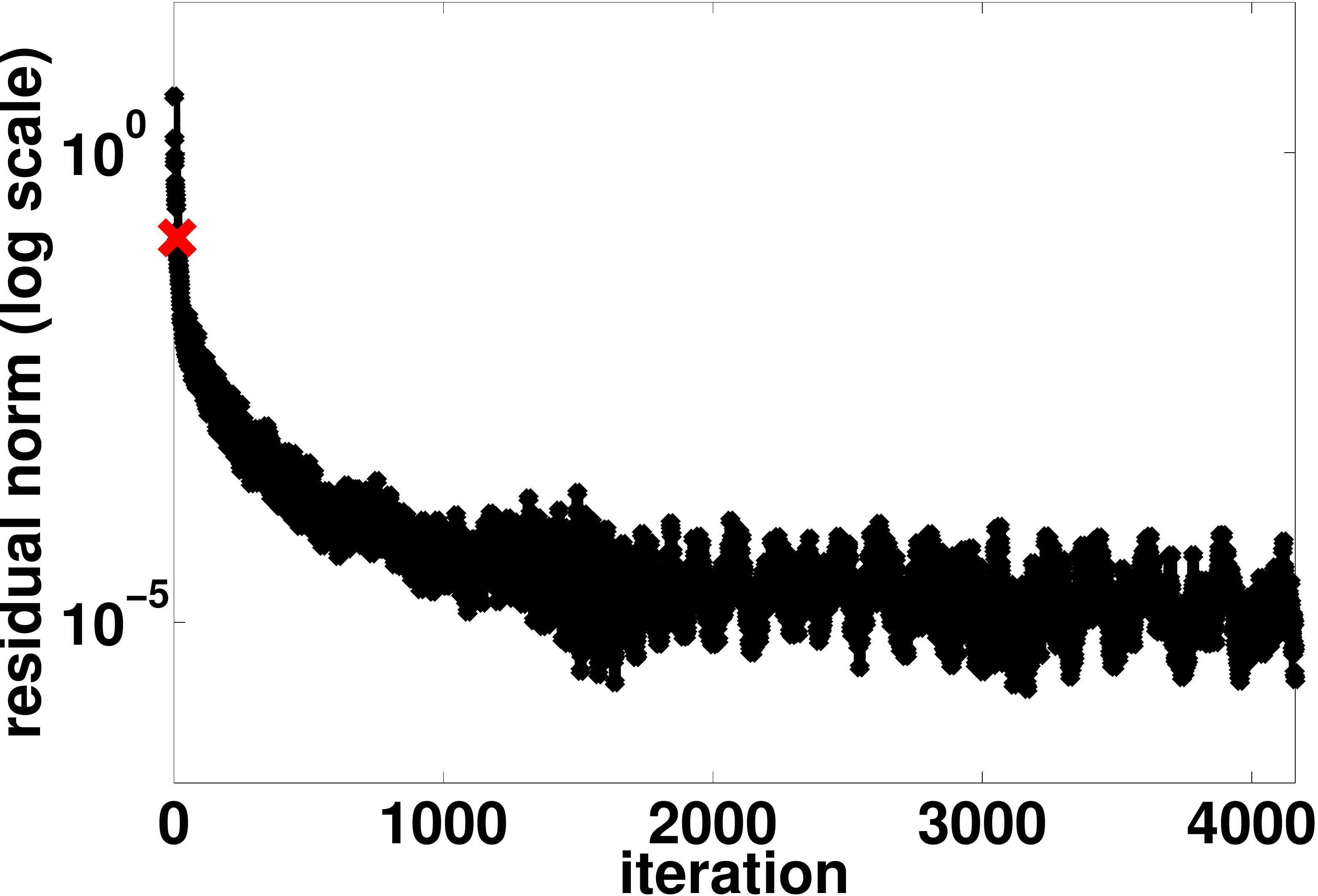}
  \caption{{{\tt mhdb416} residual norm. From left to right $k=0,1,20\%n$. The fault point is marked by the red cross.}}
  \label{figure:residual-norm-mhdb416}
\vspace{\baselineskip}

  \includegraphics[width=0.32\hsize]{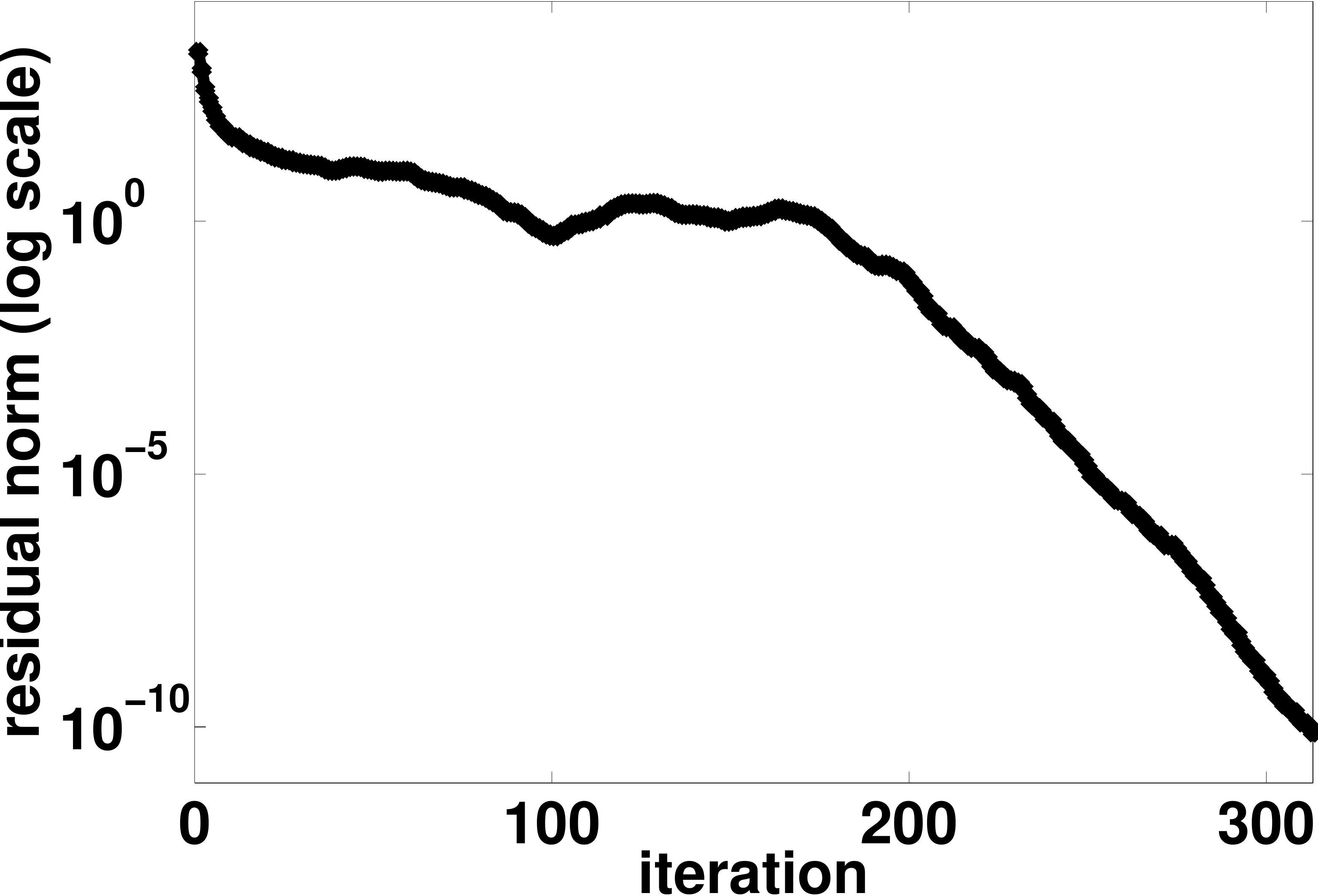}
  \includegraphics[width=0.32\hsize]{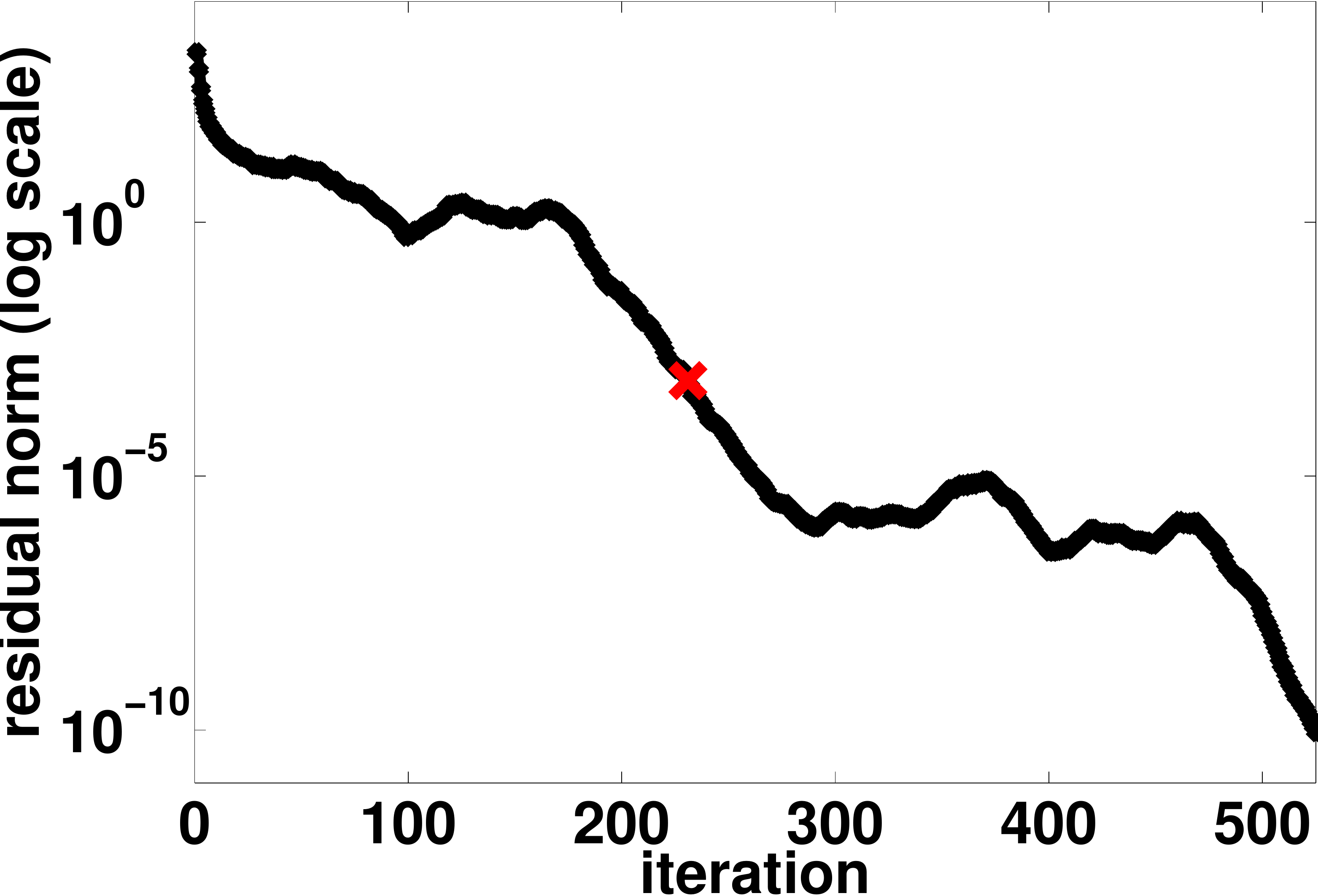}
  \includegraphics[width=0.32\hsize]{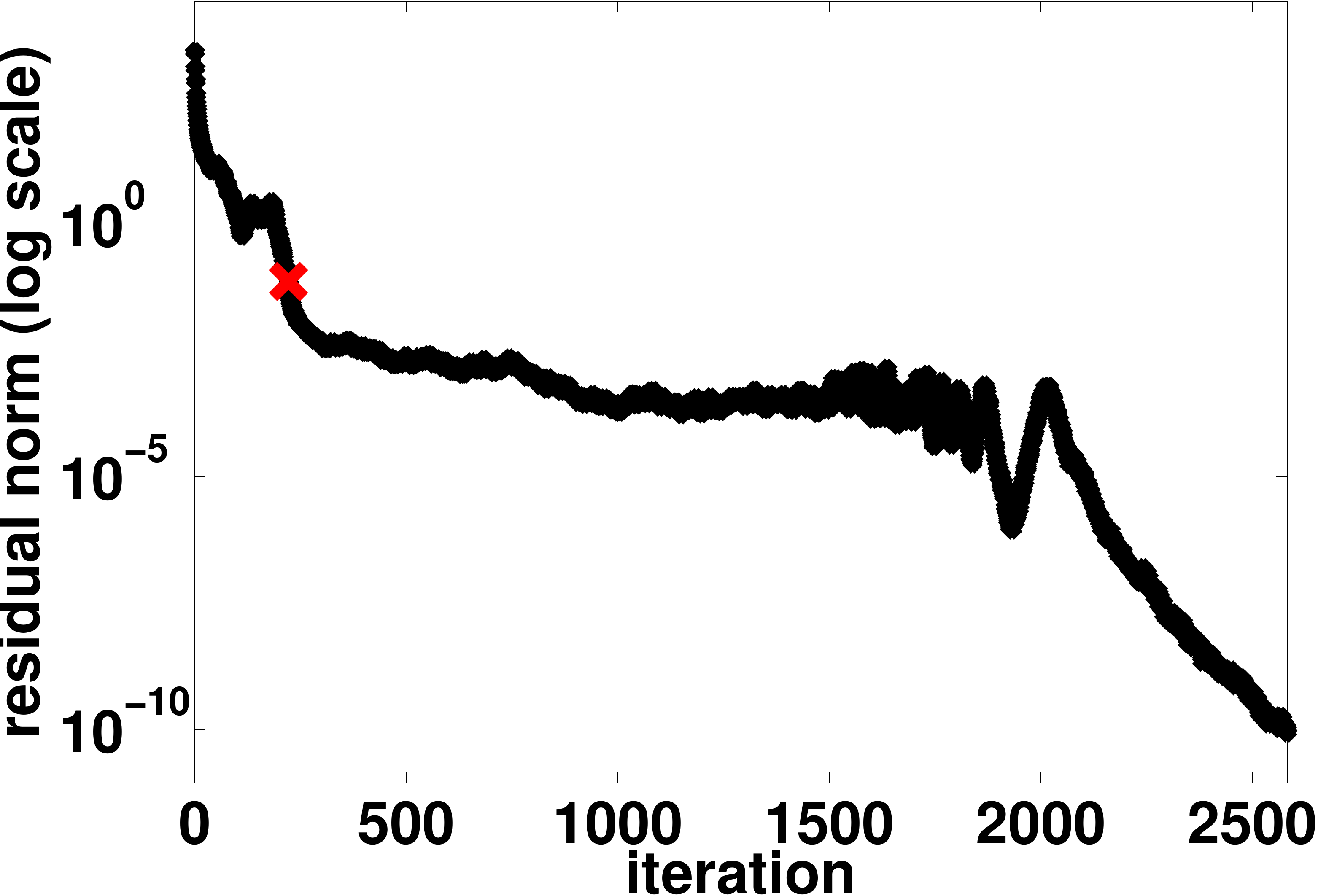}
  \caption{{{\tt nos3} residual norm. From left to right $k=0,1,20\%n$. The fault point is marked by the red cross.}}
  \label{figure:residual-norm-nos3}
\end{figure}

\begin{table}[t]
\caption{{Results on {\tt Ltridiag500}}}
\begin{center}
\begin{tabular}{lll}
  \toprule
  $k$ & Iteration number & Relative residual on raw system\\
  \midrule
  $0$ & $500$ & $1.39\times 10^{-14}$ \\
  $1$ & $540$ & $3.76\times 10^{-15}$ \\
  $20\%n$ & $2,640$ & $3.72\times 10^{-11}$\\
  \bottomrule
 \end{tabular}
\end{center}
\label{table:Ltridiag500}
\end{table}

\begin{table}[t]
\caption{{Results on {\tt mhdb416}}}
\begin{center}
\begin{tabular}{lll}
  \toprule
  $k$ & Iteration number & Relative residual on raw system\\
  \midrule
  $0$ & $4,160$ & $1.19\times 10^{-9}$ \\
  $1$ & $4,160$ & $1.47\times 10^{-5}$ \\
  $20\%n$ & $4,160$ & $2.09\times 10^{-6}$\\
  \bottomrule
 \end{tabular}
\end{center}
\label{table:mhdb416}
\end{table}

\begin{table}[t]
\caption{{Results on {\tt nos3}}}
\begin{center}
\begin{tabular}{lll}
  \toprule
  $k$ & Iteration number & Relative residual on raw system \\
  \midrule
  $0$ & $312$ & $3.51\times 10^{-14}$ \\
  $1$ & $524$ & $4.09\times 10^{-14}$ \\
  $20\%n$ & $2,581$ & $1.91\times 10^{-13}$\\
  \bottomrule
 \end{tabular}
\end{center}
\label{table:nos3}
\end{table}

To evaluate the quality of the recovered solution $\vx^{\ast}$, we compute the relative residual achieved by $\vx^{\ast}$ on the raw system \eqref{eq:raw-linear-system}: 
\[ \frac{\normof[2]{\vb- \mA \vx^{\ast}}}{\normof[2]{\vb}}.\] 
We summarize the number of iterations and relative residuals on the raw system for three test matrices in Tables \ref{table:Ltridiag500}--\ref{table:nos3}. We observe that on {\tt Ltridiag500} and {\tt nos3}, when there is only 1 fault, the erasure-coded CG can recover a solution with almost the same quality as if there is no fault. And the number of iterations are also comparable. When there are $20\%$ fault components, the erasure-coded CG can still recover an approximate solution with good quality, although the number of iterations increase substantially. In contrast to {\tt Ltridiag500} and {\tt nos3}, on {\tt mhdb416}, the erasure-coded CG reaches the maximum number of iterations (i.e., $10n$) for all the three values of $k$. Comparing the solution quality of $k=0$ to those of $k=1,20\%n$ in Table \ref{table:mhdb416} indicates more iterations are required for the latter two cases.

\section{Future work}

The current manuscript describes and documents our preliminary investigation
along the idea of erasure-coded linear solvers. To further develop the theory and
techniques for practically realizing this idea, we plan to explore the following directions.
\begin{itemize}
\item More thorough simulation and prototype tests. Our current experimental simulation
assumes the faults can only occur at the end of a CG iteration, as well as being
simultaneously. We plan to design and test under a more realistic simulation,
where the faults could happen anywhere during a CG iteration, as well as being
temporally interspersed.
\item Structure and sparse random projection design of the encoding matrix.
Our current design of the encoding matrix $\mE$ uses the random Gaussian matrix,
which is very dense. To enhance both the encoding and recovery efficiency, we plan
to adopt structure and sparse random projections~\cite{ClarksonDMMMW:2013,FoucartR:2013,MengM:2013},
which allow fast matrix-vector multiplications.
\item Improve the convergence of erasure-coded CG. Our current erasure-coded CG
adopts the truncation strategy when new faults occur. Our preliminary experiments
indicate that such a strategy may result in slow and wavy convergence when there
is a large number of faults. We plan to adapt the flexible CG technique~\cite{Notay:2000}
to our erasure-coded CG solver in order to improve its convergence speed.
\end{itemize}

\bibliographystyle{plain}
\bibliography{ErasureCodedSolver}

\end{document}